\documentclass[12pt]{amsart}

\usepackage{amsmath}
\usepackage{amssymb}
\usepackage{amsthm}
\usepackage{amscd}
\usepackage{xypic}
\usepackage{delarray}

\headsep=1cm \numberwithin{equation}{section}
\newtheorem{theorem}{Theorem}[section]
\newtheorem{lemma}[theorem]{Lemma}

\newtheorem{corollary}[theorem]{Corollary}

\theoremstyle{definition}
\newtheorem{definition}[theorem]{Definition}
\newtheorem{remark}[theorem]{Remark}

\newtheorem{convention and reminder}[theorem]{Convention and Reminder}
\newtheorem{convention and remark}[theorem]{Convention and Remark}
\newtheorem{definition and remark}[theorem]{Definition and Remark}

\newtheorem{reminders and definition}[theorem]{Reminders and Definition}

\newtheorem{notation and remarks}[theorem]{Notation and Remarks}
\newtheorem{notation and remark}[theorem]{Notation and Remark}
\newtheorem{example}[theorem]{Example}

\newcommand\Ker{\operatorname{\Ker}}

\begin{document}

\title[COMPLETE LINEAR SERIES ON A HYPERELLIPTIC CURVE]{COMPLETE LINEAR SERIES\\ ON A HYPERELLIPTIC CURVE}

\author{EUISUNG PARK}

\date{SEOUL, August 2008}



\thanks{This work was supported by the Korea Research Foundation Grant funded by the Korean
Government (KRF-C00013).}

\begin{abstract}
In this paper we study complete linear series on a hyperelliptic curve $C$ of arithmetic genus $g$. Let $A$ be the unique line bundle on $C$ such that $|A|$ is a $g^1 _2$, and let $\mathcal{L}$ be a line bundle on $C$ of degree $d$. Then $\mathcal{L}$ can be factorized as $\mathcal{L} = A^m \otimes B$ where $m$ is the largest integer satisfying $H^0 (C,\mathcal{L} \otimes A^{-m}) \neq 0$. Let $b = \mbox{deg}(B)$. We say that \textit{the factorization type of} $\mathcal{L}$ is $(m,b)$. Our main results in this paper assert that $(m,b)$ gives a precise answer for many natural questions about $\mathcal{L}$.

We first show that $(m,b)$ precisely determines the dimension of the vector spaces $H^0 (C,\mathcal{L})$ and $H^1 (C,\mathcal{L})$, and the base point freeness and the very ampleness of $\mathcal{L}$. For example, $\mathcal{L}$ is very ample if and only if $b=0$ and $m \geq g+1$ or $1 \leq b \leq g+1$ and $m+b \geq g+2$.

When $\mathcal{L}$ is very ample, we study the Hartshorne-Rao module and the minimal free resolution of the linearly normal curve embedded by $|\mathcal{L}|$. For $d=2g+1+p$, $p \geq 0$, we obtain all the graded Betti numbers explicitly. In this case, property $N_p$ holds while property $N_{p+1}$ fails to hold. We show that a finite subscheme of $C$ determined by the factorization of $\mathcal{L}$ causes the failure of $N_{p+1}$.

For $d \leq 2g$, we discuss at length the Hartshorne-Rao module and the minimal free resolution. It turns out that they are precisely determined by $(m,b)$. In particular, it is shown that the two line bundles have the same factorization type if and only if the Betti diagrams of the corresponding linearly normal curves are equal to each other. This enables us to understand how many distinct Betti diagrams occur at all.
\end{abstract}

\maketitle
\tableofcontents
\thispagestyle{empty}

\section{Introduction}
\label{1. Introduction}

\noindent In the theory of projective curves, one of the most important problem is to understand complete linear series on a curve and the maps to projective spaces defined by them. Concerning this problem, one can have a very detailed picture if either the curve in question has a very small genus, or the degree of the line bundle in question is large enough (\cite{C}, \cite{E}, \cite{F}, \cite{Green}, \cite{GL1}, \cite{GL2}, \cite{Laz}, \cite{Mum}, etc). On the other hand, it is impossible to say much about arbitrary cases. The main goal of this paper is to study the problem for line bundles on a hyperelliptic curve from various natural viewpoints.
\smallskip

Let $C$ be a hyperelliptic curve of arithmetic genus $g$, i.e. a projective integral (possibly singular) curve  such that $h^1 (C,\mathcal{O}_C) \geq 2$ and there is a degree two map $f : C \rightarrow {\mathbb P}^1$. It is well-known that $f$ is uniquely determined, up to automorphisms of ${\mathbb P}^1$. Therefore $A:= f^* \mathcal{O}_{{\mathbb P}^1} (1)$ is the unique line bundle on $C$ defining a $g^1 _2$. Also $\omega_C = A^{g-1}$ and hence $C$ is locally Gorenstein (c.f. Lemma \ref{lem:uniqueness}). One can define a natural factorization of line bundles on $C$ with respect to $A$. Let $\mathcal{L}$ be a line bundle on $C$, and consider the integer $m_{\mathcal{L}}$ defined by
\begin{equation*}
m_{\mathcal{L}} := \mbox{max}~\{ t \in {\mathbb Z} ~|~ H^0 (C,\mathcal{L} \otimes A^{-t}) \neq 0 \}
\end{equation*}
Then $\mathcal{L}$ can be factorized as $\mathcal{L}= A^{m_\mathcal{L}} \otimes B_{\mathcal{L}}$ for $B_{\mathcal{L}}:=\mathcal{L} \otimes A^{-m_\mathcal{L}}$. Let $b_{\mathcal{L}} = \mbox{deg} (B_{\mathcal{L}})$. The pair $(m_\mathcal{L} , b_\mathcal{L})$ will be called \textit{the factorization type of} $\mathcal{L}$.

Throughout this paper, we are intended to investigate the relation between the algebraic and geometric properties of $\mathcal{L}$ and the factorization type of $\mathcal{L}$. It will turn out that the pair $(m_\mathcal{L} , b_\mathcal{L})$ gives a precise answer for many natural questions about $(C,\mathcal{L})$.

In Section 2 we study a geometric meaning of the above factorization. Let $S$ be the smooth rational ruled surface ${\mathbb P} (\mathcal{O}_{{\mathbb P}^1} \oplus \mathcal{O}_{{\mathbb P}^1} (b_{\mathcal{L}}-g-1))$, and let $C_0$ and $\mathfrak{f}$ denote respectively the minimal section of $S$ and a fiber of the projection morphism $S \rightarrow {\mathbb P}^1$. Theorem~\ref{thm:surface} shows that $C$ is embedded in $S$ as a divisor linearly equivalent to $2C_0 + (2g+2-b_{\mathcal{L}})\mathfrak{f}$. Moreover, $A = \mathcal{O}_S (\mathfrak{f}) \otimes \mathcal{O}_C$ and $B_{\mathcal{L}} = \mathcal{O}_S (C_0) \otimes \mathcal{O}_C$. Therefore we can regard $\mathcal{L}$ as the restriction of the line bundle $\mathcal{O}_S (C_0 + m_{\mathcal{L}} \mathfrak{f})$ on $S$ to $C$. Since the complete linear series on $S$ are very well understood, this observation enables us to study the line bundle $\mathcal{L}$ from several viewpoints.

In Section 3 we solve the problem of the Riemann-Roch, i.e. the computation of the dimension of the $K$-vector spaces $H^0 (C,\mathcal{L})$ and $H^1 (C,\mathcal{L})$, and of the base point freeness and the very ampleness of $\mathcal{L}$. Theorem~\ref{thm:main1} shows how these properties of $\mathcal{L}$ are precisely determined by the factorization type $(m_\mathcal{L} , b_\mathcal{L})$. Also this result gives a complete list of base point free ample line bundles and very ample line bundles on $C$. Then in Theorem~\ref{thm:main2}, we provide a rounded picture of the map $\varphi_{\mathcal{L}} : C \rightarrow {\mathbb P} H^0 (C,\mathcal{L})^*$ when $\mathcal{L}$ is base point free but not very ample. As a byproduct, we obtain the classification of birationally very ample line bundles on $C$.

Through Section 4 $\sim$ Section 6, we study the linearly normal embedding of $C$. Let $\mathcal{L}$ be a very ample line bundle on $C$ of degree $d$ and with the factorization type $(m,b)$, and consider the linearly normal curve
\begin{equation*}
C \subset {\mathbb P} H^0 (C,\mathcal{L})^* ={\mathbb P}^r , ~ r=d-g,
\end{equation*}
embedded by $|\mathcal{L}|$. A natural approach to study $C \subset {\mathbb P}^r$ is to investigate the Hartshorne-Rao module, the Castelnuovo-Mumford regularity, and the graded Betti numbers of the minimal free resolution.

In Section 4  we consider the case where $d=2g+1+p$ for some $p \geq 0$. Then $C \subset {\mathbb P}^r$ is projectively normal and so $3$-regular. In Theorem~\ref{thm:main3}, we obtain all the graded Betti numbers of $C$. They depend only on $d$, and so $(m,b)$  has no effect on the form of the minimal free resolution. As a corollary, $(C,\mathcal{L})$ satisfies Green-Lazarsfeld's property $N_p$ while it fails to satisfy property $N_{p+1}$, which was first proved in \cite{GL2} when $C$ is smooth. Recall that a line bundle $L$ of degree $d=2g+1+p$ on a smooth projective curve $X$ of genus $g$ fails to satisfy property $N_{p+1}$ if and only if either $H^0 (X,L \otimes \omega_X ^{-1}) \neq 0$ and hence $X \subset {\mathbb P} H^0 (X,L)^*$ admits a $(p+3)$-secant $(p+1)$-plane or $X$ is hyperelliptic (\cite[Theorem 2]{GL2}). Thus it is an interesting problem to find a finite subscheme of $C$ which obstructs property $N_{p+1}$ of $(C,\mathcal{L})$. In Theorem~\ref{thm:geometricobstruction}, we show that the failure of property $N_{p+1}$ comes from a very special geometric property of $\mathcal{L}$.  Note that $C \subset {\mathbb P}^r$ admits a $(p+3)$-secant $(p+1)$-plane if and only if $m \geq g-1$. Thus for $m \leq g-2$, it is necessary to find a new geometric obstruction of property $N_{p+1}$. Along this line, we show that the unique effective divisor $\Gamma := C \cap C_0 \in |B_{\mathcal{L}}|$ causes the failure of property $N_{p+1}$ since
\smallskip
\begin{enumerate}
\item[a.] $\Gamma$ is contained in the rational normal curve $C_0$, and
\item[b.] $\langle \Gamma \rangle$ is a $\{(p+g-m)+(g+1-m)\}$-secant $(p+g-m)$-plane to $C$ with $g+1-m \geq 3$.
\end{enumerate}
\smallskip

\noindent For example, suppose that $d=2g+1$ and $(m,b)=(g-2,5)$. Then the corresponding linearly normal curve $C \subset {\mathbb P}^{g+1}$ has no tri-secant line while it admits a $5$-secant $2$-plane. Theorem~\ref{thm:geometricobstruction} shows that the $5$-secant $2$-plane $\langle \Gamma \rangle$ provides a geometric reason why the homogeneous ideal of $C$ cannot be generated by quadrics. Therefore the factorization of $\mathcal{L}$ is deeply related to the minimal free resolution of $C \subset {\mathbb P}^r$. This result enables us to have a coherent comprehension that the failure of property $N_{p+1}$ of $(X,L)$ is always caused by the existence of an appropriate multi-secant linear space to $X \subset {\mathbb P} H^0 (X,L)^*$.

In Section 5 we consider the case where $d \leq 2g$. Theorem~\ref{thm:lowdegree} shows that many important cohomological and homological properties of $C \subset {\mathbb P}^r$ are governed by the factorization type of $\mathcal{L}$. More precisely, the Hartshorne-Rao module and the minimal free resolution are precisely determined by $(m,b)$. In particular, the Betti diagrams of two line bundles with the same factorization type are equal to each other. This enables us to understand how many distinct Betti diagrams occur at all (Theorem~\ref{thm:DistinctBettiTable} and Remark~\ref{rem:distinctBettitable})

Finally, in Section 7 we present some examples that illustrate the results proven in Section 5 and Section 6.

\vskip 1 cm

\section{The Hyperelliptic Factorization}
\noindent Throughout this section, $C$ denotes a hyperelliptic curve of arithmetic genus $g$, i.e. $g \geq 2$ and $C$ admits a degree two map $f:C \rightarrow {\mathbb P}^1$.

Let $A:= f^* \mathcal{O}_{{\mathbb P}^1} (1)$. We begin with investigating some properties of $C$ related to $A$, which belong to folklore for smooth case.

\begin{lemma}\label{lem:uniqueness}
$(1)$ $A$ is the unique line bundle on $C$ defining $g^1 _2$.

\smallskip

\noindent $(2)$ $\omega_C =A^{g-1}$ and hence $C$ is locally Gorenstein.
\end{lemma}

\begin{proof}
(1) Let $\kappa : \widetilde{C} \rightarrow C$ be the normalization of $C$, and let $\rho$ be the genus of $\widetilde{C}$. Note that $f \circ \kappa : \widetilde{C} \rightarrow {\mathbb P}^1$ is a degree two morphism. Let $A'$ be a line bundle on $C$ satisfying $\mbox{deg}(A')=h^0 (C,A')=2$, and let $h:C \rightarrow {\mathbb P}^1$ be the morphism defined by $A'$. Consider the following commutative diagram:
\begin{equation*}
\begin{CD}
\widetilde{C} & \quad \quad \stackrel{\kappa}{\rightarrow} \quad & C \\
  & h \circ \kappa   \searrow \quad &  \downarrow h            \\
  &                               &   {\mathbb P}^1
\end{CD}
\end{equation*}
To verify that $A=A'$, it suffices to show that $h \circ \kappa = f \circ \kappa$. If $\rho \geq 2$, then $\widetilde{C}$ is a smooth hyperelliptic curve and hence $h \circ \kappa = f \circ \kappa$. If $\rho =1$, let $x \in \mbox{Sing}(C)$. Then
\begin{equation*}
(h \circ \kappa)^* \mathcal{O}_{{\mathbb P}^1} (1)  = (\kappa ^* \mathcal{I}_{x/C})^{-1} = (f \circ \kappa)^* \mathcal{O}_{{\mathbb P}^1} (1)
\end{equation*}
and hence $h \circ \kappa = f \circ \kappa$. If $\rho=0$, let $\widetilde{C} \subset {\mathbb P}^2$ be the degree two embedding. Then degree two maps $\widetilde{C} \rightarrow {\mathbb P}^1$ is parameterized by ${\mathbb P}^2 \setminus \widetilde{C}$. Now let $x,y \in \mbox{Sing}(C)$. Then $(\kappa ^* \mathcal{I}_{x/C})^{-1}$ and $(\kappa ^* \mathcal{I}_{y/C})^{-1}$ determines distinct two lines which meet at a point $P \in {\mathbb P}^2$. Therefore $h \circ \kappa$ should be the linear projection of $\widetilde{C} \subset {\mathbb P}^2$ from $P$. In particular, it is uniquely determined by $\kappa$.

\smallskip

\noindent (2) Since $h^0 (C,A^{g-1}) \geq g$, we have
\begin{equation*}
h^1 (C,A^{g-1})=h^0 (\omega_C \otimes A^{-g+1})>0.
\end{equation*}
This shows that $\omega_C \otimes A^{-g+1} = \mathcal{O}_C$ since the degree of the quasi-invertible sheaf $\omega_C \otimes A^{-g+1}$ is zero. Therefore $\omega_C =A^{g-1}$ and $C$ is locally Gorenstein.
\end{proof}

Now we introduce to a natural factorization of line bundles on $C$ with respect to $A$. It will turn out through the remaining sections that all the algebraic and geometric properties of complete linear series on $C$ are explicitly determined by the type of this factorization.

\smallskip

\begin{definition}
Let $\mathcal{L}$ be a line bundle on $C$.

\smallskip

\noindent (1) The integer $m_{\mathcal{L}} := \mbox{max}~\{ t \in {\mathbb Z} ~|~ H^0 (C,\mathcal{L} \otimes A^{-t}) \neq 0 \}$ is called the multiplicity of $\mathcal{L}$. If $m_{\mathcal{L}} =0$, then we say that $\mathcal{L}$ is normalized.

\smallskip

\noindent (2) The line bundle $B_{\mathcal{L}} := \mathcal{L} \otimes A^{-m_{\mathcal{L}}}$
is called the normalized part of $\mathcal{L}$.

\smallskip

\noindent (3) The integer $b_{\mathcal{L}} := \mbox{deg}(\mathcal{L} \otimes A^{-m_{\mathcal{L}}})$
is called the normalized degree of $\mathcal{L}$.

\smallskip

\noindent (4) $\mathcal{L} =   A^{m_{\mathcal{L}}} \otimes B_{\mathcal{L}}$ is said to be the hyperelliptic factorization of $\mathcal{L}$.

\smallskip

\noindent (5) We say that $(m_{\mathcal{L}},b_{\mathcal{L}})$ is the factorization type of $\mathcal{L}$.
\end{definition}

\smallskip

\noindent The hyperelliptic factorization of $\mathcal{L}$ is unique in the sense that if $\mathcal{L} =   A^m \otimes B$ where $B$ is normalized, then $m = m_{\mathcal{L}}$ and $B = B_{\mathcal{L}}$. Also the normalized degree of $\mathcal{L}$ should satisfy $0 \leq b_{\mathcal{L}} \leq g+1$. On the other hand, for any integer $2 \leq b \leq g+1$, there exists a normalized line bundle of degree $b$. More precisely, consider smooth points $P_1 , \cdots , P_b$ of $C$ such that no two of the $P_i$'s are conjugate under the hyperelliptic involution. Then $\mathcal{O}_C (P_1 + \cdots + P_b)$ is a normalized line bundle of degree $b$ on $C$.

A crucial geometric meaning of this factorization is provided by

\smallskip
\begin{theorem}\label{thm:surface}
Let $B \in \mbox{Pic}C$ be a normalized line bundle of degree $b$, and let $S$ be the smooth rational ruled surface ${\mathbb P} (\mathcal{O}_{{\mathbb P}^1} \oplus \mathcal{O}_{{\mathbb P}^1} (b-g-1))$. Then there is an embedding of $C \subset S$ satisfying the following three conditions:

\smallskip

\begin{enumerate}
\item[(a)] $A = \mathcal{O}_S (\mathfrak{f}) \otimes \mathcal{O}_C$;
\item[(b)] $B = \mathcal{O}_S (C_0) \otimes \mathcal{O}_C$; and
\item[(c)] $C$ is linearly equivalent to $2C_0 + (2g+2-b)\mathfrak{f}$ as a divisor of $S$.
\end{enumerate}

\smallskip

\noindent Here $C_0$ and $\mathfrak{f}$ denote respectively the minimal section of $S$ and a fiber of the projection map $S \rightarrow {\mathbb P}^1$.
\end{theorem}

\smallskip

\begin{proof}
Consider the line bundle $\mathcal{L}:=A^{g+2} \otimes B$ on $C$ of degree $d= 2g+4+b$. Since $b \geq 0$,
$\mathcal{L}$ a non-special very ample line bundle, and that the linearly normal curve
\begin{equation*}
C \subset {\mathbb P} H^0 (C,\mathcal{L})^* = {\mathbb P}^r , \quad r=g+b+4,
\end{equation*}
embedded by the complete linear series $|\mathcal{L}|$ is cut out by quadrics (cf. \cite{F}). In particular, $C \subset {\mathbb P}^r$ does not admit a tri-secant line.

\smallskip

We first construct a smooth rational normal surface scroll $S \subset {\mathbb P}^r$ which contains $C$. This can be done by Eisenbud-Koh-Stillman's method in \cite{EKS}. One can also find the details in \cite[Chapter VI]{E}. Consider the multiplication map
\begin{equation*}
\mu : H^0 (C,A) \otimes H^0 (C,\mathcal{L}\otimes A^{-1}) \rightarrow H^0 (C,\mathcal{L}).
\end{equation*}
If we choose bases $\{e_1, e_2\}$ and $\{f_1 , \cdots, f_{r-1} \}$ of $H^0 (C,A)$ and $H^0 (C,\mathcal{L}\otimes A^{-1})$, respectively, then the $2 \times (r-1)$ matrix
\begin{equation*}
M(A,\mathcal{L} \otimes A^{-1}) := ( \mu (e_i \otimes f_j))
\end{equation*}
can be regarded as a matrix of linear forms on ${\mathbb P}^r$. This is a $1$-generic matrix and hence its $2 \times 2$ minors define a rational normal surface scroll $S \subset {\mathbb P}^r$ which contains $C$. Also $S$ is smooth since $\mu$ is surjective. Geometrically, $S$ can be defined as
\begin{equation*}
S = \bigcup_{y \in {\mathbb P}^1} <f^{-1} (y)> \subset {\mathbb P}^r
\end{equation*}
where $f : C \rightarrow {\mathbb P}^1$ is the degree two map.

\smallskip

Now we will determine the numerical type of $S$ and the divisor class of $C$ in $S$. Since $S$ is a smooth rational ruled surface, there is an integer $e \geq 0$ such that
\begin{equation*}
S = {\mathbb P} (\mathcal{O}_{{\mathbb P}^1} \oplus \mathcal{O}_{{\mathbb P}^1} (-e)).
\end{equation*}
Let $C_0$ and $\mathfrak{f}$ be respectively the minimal section of $S$ and a fiber of the projection map $\pi : S \rightarrow {\mathbb P}^1$. Then $C \equiv kC_0 + \ell \mathfrak{f}$ for some $k \geq 1$ and $\ell \in {\mathbb Z}$. Since $C$ is not a smooth rational curve, $k \geq 2$. On the other hand, the rulings of $S$ are $k$-secant lines to $C$. So $C \subset {\mathbb P}^r$ cannot be cut out by quadrics if $k \geq 3$. In conclusion, $k=2$ and so the restriction of $\pi : S \rightarrow {\mathbb P}^1$ to $C$ is a degree two map. Since this restriction map is defined by $\mathcal{O}_S (\mathfrak{f})\otimes \mathcal{O}_C$, Lemma \ref{lem:uniqueness}.$(1)$ guarantees that $\mathcal{O}_S (\mathfrak{f}) \otimes \mathcal{O}_C = A$, which completes the proof of (a).

\smallskip

To verify (b) and (c), write the hyperplane bundle $\mathcal{O}_S (1)$ of $S \subset {\mathbb P}^r$ as $\mathcal{O}_S (C_0 + n \mathfrak{f})$, $n \in {\mathbb Z}$. We will show that $n=g+2$. Since $\mathcal{O}_S (\mathfrak{f}) \otimes \mathcal{O}_C = A$ and $\mathcal{O}_S (1) \otimes \mathcal{O}_C = \mathcal{L}$, the exact sequence
\begin{equation*}
0 \rightarrow \mathcal{O}_S (-C) \rightarrow \mathcal{O}_S \rightarrow \mathcal{O}_C \rightarrow 0,
\end{equation*}
induces the isomorphism

\smallskip

\renewcommand{\descriptionlabel}[1]%
             {\hspace{\labelsep}\textrm{#1}}
\begin{description}
\setlength{\labelwidth}{13mm}
\setlength{\labelsep}{1.5mm}
\setlength{\itemindent}{0mm}

\item[(2.1)] $H^i (S,\mathcal{O}_S (C_0 + (n-j) \mathfrak{f})) \cong H^i (C,A^{g+2-j} \otimes B)$
\end{description}

\smallskip

\noindent for $i \in \{0,1 \}$ and all $j \in {\mathbb Z}$. In particular,
\begin{equation*}
H^0 (S,\mathcal{O}_S (C_0 + (n-g-2) \mathfrak{f})) \cong H^0 (C, B) \neq 0
\end{equation*}
while
\begin{equation*}
H^0 (S,\mathcal{O}_S (C_0 + (n-g-3) \mathfrak{f})) \cong H^0 (C,A^{-1} \otimes B)=0.
\end{equation*}
This implies that $n=g+2$. Now we have
\begin{equation*}
\mathcal{O}_S (C_0) \otimes \mathcal{O}_C = \mathcal{L} \otimes A^{-g-2} = B .
\end{equation*}
The value of $e$ is determined by using the degree of $S \subset {\mathbb P}^r$. Indeed
\begin{equation*}
\mbox{deg}(S)= g+3+b = (C_0 + (g+2)\mathfrak{f})^2 = -e+2(g+2)
\end{equation*}
which shows that $e=g+1-b$. Finally, we determine the value of $\ell$ from the equality
$\mbox{deg}(\mathcal{L})= (C_0 + (g+2)\mathfrak{f}).(2C_0 +\ell \mathfrak{f})$, where
\begin{equation*}
\mbox{deg}(\mathcal{L})=2g+4+b \quad \mbox{and} \quad (C_0 + (g+2)\mathfrak{f}).(2C_0 +\ell \mathfrak{f})=-2e+\ell+2(g+2).
\end{equation*}
Then we have $\ell = b+2e = 2g+2-b$.
\end{proof}

\vskip 1 cm

\section{Base Point Freeness And Very Ampleness}
\noindent Throughout this section we keep the previously introduced notation. In particular, $C$ denotes a hyperelliptic curve of arithmetic genus $g$.

The main purpose of this section is to prove

\begin{theorem}\label{thm:main1}
Let $\mathcal{L}$ be a line bundle on $C$ with the factorization type $(m,b)$. Then

\smallskip

\begin{enumerate}
\item[(1)] For each $i \in \{ 0,1 \}$,
\begin{equation*}
h^i (C,\mathcal{L}) = h^i ({\mathbb P}^1,\mathcal{O}_{{\mathbb P}^1} (m)) + h^i ({\mathbb P}^1,\mathcal{O}_{{\mathbb P}^1} (m+b-g-1)).
\end{equation*}
In particular, $\mathcal{L}$ is non-special if and only if $m+b \geq g$.
\smallskip
\item[(2)] $\mathcal{L}$ is base point free if and only if either \smallskip
\begin{enumerate}
\item[a.] $b=0$ $($and hence $B_{\mathcal{L}} =\mathcal{O}_C)$ and $m \geq 0$ or
\item[b.] $1 \leq b \leq g+1$ and $m +b \geq g+1$.
\end{enumerate}

\smallskip

\item[(3)] $\mathcal{L}$ is very ample if and only if either \smallskip
\begin{enumerate}
\item[a.] $b=0$ $($and hence $B_{\mathcal{L}}=\mathcal{O}_C)$ and $m \geq g+1$ or
\item[b.] $b=1$ $($and hence $B_{\mathcal{L}}=\mathcal{O}_C (P)$ for some $P \in C)$ and $m \geq g$ or
\item[c.] $2 \leq b \leq g+1$ and $m +b \geq g+2$.
\end{enumerate}
\end{enumerate}
\end{theorem}
\smallskip

\begin{proof}
Let $S$ be the rational ruled surface ${\mathbb P} (\mathcal{O}_{{\mathbb P}^1} \oplus \mathcal{O}_{{\mathbb P}^1} (b-g-1))$ with the minimal section  $C_0$ and a ruling $\mathfrak{f}$. By Theorem \ref{thm:surface}, $C$ is contained in $S$ as a divisor linearly equivalent to $2C_0 +(2g+2-b)\mathfrak{f}$. Moreover, $A = \mathcal{O}_S (\mathfrak{f}) \otimes \mathcal{O}_C$ and $b = \mathcal{O}_S (C_0) \otimes \mathcal{O}_C$. In particular, $\mathcal{L}$ is equal to the restriction of the line bundle $\mathcal{O}_S (C_0 + m \mathfrak{f})$ on $S$ to $C$.

\smallskip

\noindent (1) By (2.1), we have
\begin{equation*}
\begin{CD}
h^i (C,\mathcal{L}) & \quad = \quad & h^i (C,A^{m} \otimes b) \quad \quad \quad \quad \quad \quad \quad \quad \quad \quad \quad \quad \quad \quad\\
                    & \quad = \quad & h^i (S,\mathcal{O}_S (C_0+m \mathfrak{f})) \quad \quad \quad \quad \quad \quad \quad \quad \quad \quad \quad \quad\\
                    & \quad = \quad & h^i ({\mathbb P}^1,\mathcal{O}_{{\mathbb P}^1} (m)) + h^i ({\mathbb P}^1,\mathcal{O}_{{\mathbb P}^1} (m+b-g-1))
\end{CD}
\end{equation*}
for all $i \in \{0,1 \}$.

\smallskip

\noindent (2) Remember that $0 \leq b \leq g+1$. If $b =0$ and hence $B=\mathcal{O}_C$, then $\mathcal{L}= A^{m}$ is base point free if and only if $m \geq 0$.
Now assume that $b \geq 1$. Since $H^0 (C,\mathcal{L}) \cong H^0 (S,\mathcal{O}_S (C_0 + m \mathfrak{f}))$, we have
\begin{equation*}
\mbox{Bs}~|\mathcal{L}| = C \cap \mbox{Bs}~|\mathcal{O}_S (C_0 + m \mathfrak{f})|.
\end{equation*}
For $m < g+1-b$, $C_0 \subset \mbox{Bs}~|\mathcal{O}_S (C_0 + m \mathfrak{f})|$ and hence $C \cap C_0 \subset \mbox{Bs}~|\mathcal{L}|$. This shows that $\mathcal{L}$ fails to be base point free since $\mbox{length}~(C \cap C_0)=b \geq 1$.
For $m \geq g+1-b$, $\mbox{Bs}~|\mathcal{O}_S (C_0 + m \mathfrak{f})| =\emptyset$ and so $\mathcal{L}$ is base point free.

\smallskip

\noindent (3) We need to classify all very ample line bundles among base point free line bundles. For $m \geq g+2-b$, $\mathcal{O}_S (C_0 +m \mathfrak{f})$ is a very ample line bundle on $S$, and so $\mathcal{L}$ is a very ample line bundle on $C$.
Now assume that $m \leq g+1-b$. By $(2)$, $\mathcal{L}$ is a base point free line bundle if and only if either

\smallskip

\begin{enumerate}
\item[$(\alpha)$] $b=0$ and $0 \leq m \leq g+1$ or else
\item[$(\beta)$] $1 \leq b \leq g+1$ and $m=g+1-b$.
\end{enumerate}

\smallskip

\noindent Clearly $A^m$ is very ample if and only if $m \geq g+1$. Also for the case $(\beta)$, the line bundle $\mathcal{O}_S (C_0 + (g+1-b) \mathfrak{f})$ is base point free, and defines an embedding of $S \setminus C_0$. Furthermore $C_0$ maps to a point. Observe that $\mbox{length}~(C \cap C_0)=b$. Since $H^0 (S,\mathcal{O}_S (C_0 + m \mathfrak{f})) \cong H^0 (C,\mathcal{L})$, we can conclude that $\mathcal{L}$ is very ample if and only if $b=1$, which completes the proof of (3).
\end{proof}

\smallskip

\begin{corollary}\label{cor:RR}
Let $B \in \mbox{Pic}C$ be a normalized line bundle of degree $b$. Then
$$h^0 (C,B)= \begin{cases} 1 & \mbox{for $0 \leq b \leq g$, and}\\
                           2 & \mbox{for $b=g+1$}. \end{cases} $$
\end{corollary}

\smallskip

\begin{proof}
By Theorem \ref{thm:main1}.(1), $h^0 (C,B)=1+h^0 ({\mathbb P}^1,\mathcal{O}_{{\mathbb P}^1} (b-g-1))$.
\end{proof}

\smallskip

\begin{corollary}\label{cor:NotVeryample}
Let $\mathcal{L}$ be a base point free ample line bundle on $C$ with the hyperelliptic factorization type $(m,b)$. Then $\mathcal{L}$ fails to be very ample if and only if
\begin{enumerate}
\item[($\alpha$)] $\mathcal{L} =A^m$ for some $1 \leq m \leq g$;
\item[($\beta$)] $2 \leq b \leq g$ $($and hence $h^0 (C,b)=1$$)$ and $m + b =g+1$;
\item[($\gamma$)] $\mathcal{L}$ is a normalized line bundle of degree $g+1$.
\end{enumerate}
\end{corollary}

\smallskip

\begin{proof}
The assertion comes immediately by Theorem \ref{thm:main1}.(2) and (3).
\end{proof}

\smallskip

We conclude this section by investigating the maps of $C$ to projective spaces defined by $\mathcal{L}$ in Corollary \ref{cor:NotVeryample}.($\alpha$) $\sim$ ($\gamma$). We obtain the following

\smallskip

\begin{theorem}\label{thm:main2}
Let $\mathcal{L}$ be a base point free ample line bundle on $C$, defining a morphism $\varphi_{\mathcal{L}} : C \rightarrow {\mathbb P}^r$, $r=h^0 (C,\mathcal{L})-1$. Then

\smallskip

\begin{enumerate}
\item[$(\alpha)$] If $\mathcal{L}=A^m$ for some $1 \leq m \leq g$, then $r=m$ and $\varphi$ consists of the double covering $f : C \rightarrow {\mathbb P}^1$ followed by the $m$-uple embedding of ${\mathbb P}^1$ in ${\mathbb P}^m$. In particular, the image $\varphi_{\mathcal{L}} (C) \subset {\mathbb P}^m$ is a rational normal curve of degree $m$.\smallskip
\item[$(\beta)$] Assume that $2 \leq b \leq g$ $($and hence $h^0 (C,b)=1)$ and $m(\mathcal{L})+b(\mathcal{L})=g+1$.
    Let $P_1 + \cdots +P_{b }$ be the unique divisor in $|b|$. Then $\mathcal{L}$ is nonspecial and the morphism $\varphi_{\mathcal{L}} : C \rightarrow {\mathbb P}^{d-g}$, $d=\mbox{deg}(\mathcal{L})$, is birational onto its image. More precisely, $\varphi_{\mathcal{L}}$ maps $\{P_1,\cdots,P_{b }\}$ to a point $q \in {\mathbb P}^{d-g}$ and the restriction map
\begin{equation*}
\varphi_{\mathcal{L}}\upharpoonright_{C \setminus \{P_1,\cdots,P_{ b} \}}: C \setminus \{P_1,\cdots,P_{ b} \} \rightarrow \varphi_{\mathcal{L}} (C) \setminus \{q\}
\end{equation*}
is an isomorphism.\smallskip
\item[$(\gamma)$] If $\mathcal{L}$ is a normalized line bundle of degree $g+1$, then $r=1$ and the map $\varphi: C \rightarrow {\mathbb P}^1$ is a $(g+1)$-fold covering.
\end{enumerate}
\end{theorem}

\smallskip

\begin{proof}
For $(\alpha)$, note that $\mathcal{L} = f^* \mathcal{O}_{{\mathbb P}^1} (m)$ and also $H^0 (C,\mathcal{L}) = f^* H^0 ({\mathbb P}^1,\mathcal{O}_{{\mathbb P}^1} (m))$ since $m \leq g$. This completes the proof.

For $(\beta)$, let us remind of the proof of Theorem \ref{thm:main1}.$(3)$. Indeed $\mathcal{L}$ is the restriction of the line bundle $\mathcal{O}_S (C_0 + (g+1-b) \mathfrak{f})$ on $S$, which is base point free and defines an embedding of $S \setminus C_0$ to ${\mathbb P}^r$. Furthermore $C_0$ maps to a point, say $q \in {\mathbb P}^r$. Since $\varphi_{\mathcal{L}} ^{-1} (q) = C \cap C_0 = \{P_1,\cdots,P_{ b} \}$, we have the desired isomorphism
\begin{equation*}
C \setminus \{P_1,\cdots,P_{b} \} \cong \varphi_{\mathcal{L}} (C) \setminus \{q\}.
\end{equation*}

For $(\gamma)$, Theorem \ref{thm:surface} asserts that $C$ is contained in $S \cong {\mathbb P}_1 ^1 \times {\mathbb P}_2 ^1$ as a divisor linearly equivalent to $2C_0 + (g+1) \mathfrak{f}$ and $\mathcal{L} = \mathcal{O}_S (C_0 ) \otimes \mathcal{O}_C$. Since the morphism $\varphi_{\mathcal{L}}: C \rightarrow {\mathbb P}^1$ is the restriction of the second projection map $S \rightarrow {\mathbb P}_2 ^1$ to $C$, it is a $(g+1)$-fold covering.
\end{proof}

\vskip 1 cm

\section{Minimal Free Resolution I - High Degree}
\noindent We now wish to give a precise description of the minimal free resolution of a linearly normal hyperelliptic curve $C$ of arithmetic genus $g$ and of degree $d \geq 2g+1$. The case where $d \leq 2g$ will be dealt with in the next section.

Throughout this section, let $\mathcal{L}$ be a line bundle on $C$ of degree $d \geq 2g+1$. Thus $\mathcal{L}$ is non-special and very ample, and defines a linearly normal embedding
\begin{equation*}
C \subset {\mathbb P} H^0 (C,\mathcal{L})^* ={\mathbb P}^r,\quad r=d-g.
\end{equation*}
Let $R$ be the homogeneous coordinate ring of ${\mathbb P}^r$ and let $I_C$ be the vanishing ideal of $C$.
By a result in \cite{F}, $C$ is projectively normal and $3$-regular. Thus a minimal free resolution of $I_C$ over $R$ is of the form
\begin{equation*}
0 \rightarrow F_{r-1} \rightarrow \cdots \rightarrow F_i \rightarrow \cdots \rightarrow F_1 \rightarrow I_C \rightarrow 0,
\end{equation*}
where $F_i = R(-i-1)^{\beta_{i,1}} \oplus R(-i-2)^{\beta_{i,2}}$.

Our main result in this section is

\smallskip

\begin{theorem}\label{thm:main3}
Under the situation just state, suppose that $d=2g+1+p$. Then

\smallskip

\begin{enumerate}
\item[(a)] $\beta_{1,1} = {{r-1} \choose {2}}$,
\item[(b)] $\beta_{1,2}=\cdots=\beta_{p,2}=0$ and $\beta_{i,2} = (i-p){{r-1} \choose {i}}$ for $p+1 \leq i \leq r-1$,
\item[(c)] $\beta_{i+1,1}=\beta_{i,2}-g{{r-1}\choose{i}}+(r-1){{r-1}\choose{i+1}}-{{r-1}\choose{i+2}}$ for all $1 \leq i \leq r-2$.
\end{enumerate}

\smallskip

\noindent Therefore $(C,\mathcal{L})$ satisfies property $N_p$ but fails to satisfy property $N_{p+1}$.
\end{theorem}

\smallskip

\begin{proof}
(a) The assertion comes immediately from the $2$-normality of $C$.

\smallskip

\noindent (b) Let $\mathcal{I}_C$ be the sheaf of ideals of $C$ and let $\mathcal{M}=\Omega_{{\mathbb P}^r}(1)$. Then
\smallskip
\begin{enumerate}
\item[(4.1)] \quad \quad \quad $\beta_{i,2}= h^1 ({\mathbb P}^r,\bigwedge^i \mathcal{M} \otimes \mathcal{I}_C (2))$.
\end{enumerate}
\smallskip
\noindent Let $(m,b)$ be the hyperelliptic factorization type of $\mathcal{L}$. Then Theorem \ref{thm:surface} shows that $C$ is contained in $S={\mathbb P} (\mathcal{O}_{{\mathbb P}^1} \oplus \mathcal{O}_{{\mathbb P}^1} (b-g-1))$. Furthermore,
for the minimal section $C_0$ of $S$ and a fiber $\mathfrak{f}$ of the projection map $\pi : S \rightarrow {\mathbb P}^1$, the following three conditions hold:
\begin{equation*}
\mathcal{O}_S (\mathfrak{f}) \otimes \mathcal{O}_C = A, \quad \mathcal{O}_S (C_0) \otimes \mathcal{O}_C = B, \quad \mbox{and} \quad  C \equiv 2C_0 + (2g+2-b)\mathfrak{f}
\end{equation*}
In particular, $\mathcal{L}$ is equal to the restriction of $L:=\mathcal{O}_S (C_0 + m \mathfrak{f})$ to $C$. Remember that $L$ is a very ample line bundle on $S$ except the following cases:

\smallskip

\begin{enumerate}
\item[($\alpha$)] $b=0$ (and hence $B=\mathcal{O}_C$) and $m=g+1$
\item[($\beta$)] $b=1$ (and hence $B=\mathcal{O}_C (P)$ for some $P \in C$) and $m=g$
\end{enumerate}

\smallskip

\noindent We first consider the case where $L$ is very ample. Then the two exceptional cases ($\alpha$) and ($\beta$) will be treated in turn.

Suppose that $L$ is a very ample line bundle on $S$. Then it defines a smooth rational normal surface scroll $S \subset {\mathbb P}^r$. Let $\mathcal{I}_S$ be the sheaf of ideals of $S$, and consider the short exact sequence
\smallskip
\begin{enumerate}
\item[(4.2)] \quad \quad \quad $0 \rightarrow \mathcal{I}_S \rightarrow \mathcal{I}_C \rightarrow \mathcal{O}_S (-C) \rightarrow 0$.
\end{enumerate}
\smallskip
This induces the cohomology long exact sequence
\smallskip
\begin{enumerate}
\item[(4.3)] \quad $H^1 ({\mathbb P}^r,\bigwedge^i \mathcal{M} \otimes \mathcal{I}_S (j)) \rightarrow H^1 ({\mathbb P}^r,\bigwedge^i \mathcal{M} \otimes \mathcal{I}_C (j))$
\item[] \quad \quad \quad $\rightarrow H^1 (S,\bigwedge^i \mathcal{M}_{S} \otimes \mathcal{O}_S (-C) \otimes L^j) \rightarrow H^2 ({\mathbb P}^r,\bigwedge^i \mathcal{M} \otimes \mathcal{I}_S (j))$
\end{enumerate}
\smallskip
where $\mathcal{M}_S$ is the restriction of $\mathcal{M}$ to $S$. Also since $S \subset {\mathbb P}^r$ is $2$-regular,
\smallskip
\begin{enumerate}
\item[(4.4)] \quad $H^1 ({\mathbb P}^r,\bigwedge^i \mathcal{M} \otimes \mathcal{I}_S (j)) = H^2 ({\mathbb P}^r,\bigwedge^i \mathcal{M} \otimes \mathcal{I}_S (j))=0$ for all $i \geq 1$.
\end{enumerate}
\smallskip
\noindent Thus (4.1), (4.3) and (4.4) imply that
\smallskip
\begin{enumerate}
\item[(4.5)]  \quad \quad $\beta_{i,2}= h^1 (S,\bigwedge^i \mathcal{M}_{S} \otimes \mathcal{O}_S (-C) \otimes L^2)$ for all $i \geq 1$.
\end{enumerate}
\smallskip

\noindent Let $\mathcal{E}$ be $\pi_* \mathcal{O}_S (1)$, i.e. $\mathcal{E}=\mathcal{O}_{{\mathbb P}^1} (m) \oplus \mathcal{O}_{{\mathbb P}^1} (m+b-g-1)$, and let $\mathcal{M}_{\mathcal{E}}$ denote the kernel of the surjective homomorphism $H^0 ({\mathbb P}^1 , \mathcal{E}) \otimes {\mathbb P}^1 \rightarrow \mathcal{E}$.
Note that $\mathcal{M}_\mathcal{E} = \mathcal{O}_{{\mathbb P}^1} (-1)^{\oplus (r-1)}$. Letting $M=\mathcal{O}_S (-C_0 +(m+b-g-1)\mathfrak{f})$, we have the following commutative diagram:
\smallskip
\begin{equation*}
\begin{CD}
&  &&  && 0 &\\
&  &&  && \downarrow &\\
& 0 && && M &\\
& \downarrow &&  && \downarrow &\\
0 \rightarrow &\pi^* \mathcal{M}_\mathcal{E}& \rightarrow & \pi^* H^0 ({\mathbb P}^1,\mathcal{E}) \otimes \mathcal{O}_S  & \rightarrow  & \pi^* \mathcal{E} & \rightarrow 0 \\
& \downarrow && \parallel && \downarrow &\\
0 \rightarrow & \mathcal{M}_S & \rightarrow & H^0 (S,L) \otimes \mathcal{O}_S & \rightarrow & L & \rightarrow 0 \\
& \downarrow &&  && \downarrow &\\
& M && && 0  &\\
& \downarrow  &&  && &\\
& 0 &&  && & \\
\end{CD}
\end{equation*}
\smallskip
\noindent In particular, the first column gives the short exact sequence
\smallskip
\begin{enumerate}
\item[(4.6)]  \quad \quad $0 \rightarrow \bigwedge^i \pi^* \mathcal{M}_{\mathcal{E}} \rightarrow \bigwedge^i \mathcal{M}_{S} \rightarrow \bigwedge^{i-1} \pi^* \mathcal{M}_{\mathcal{E}} \otimes M \rightarrow 0$
\end{enumerate}
\smallskip
for all $i \geq 1$, which induces the cohomology long exact sequence
\smallskip
\begin{enumerate}
\item[(4.7)] \quad $\cdots \rightarrow H^0 (S,\bigwedge^{i-1} \pi^* \mathcal{M}_\mathcal{E} \otimes M \otimes \mathcal{O}_S (-C) \otimes L^2) $ \smallskip
\item[] \quad \quad \quad  \quad \quad $\rightarrow H^1 (S,\bigwedge^i \pi ^* \mathcal{M}_\mathcal{E} \otimes \mathcal{O}_S (-C) \otimes L^2)$ \smallskip
\item[] \quad \quad \quad  \quad \quad $\rightarrow H^1 (S,\bigwedge^i \mathcal{M}_S \otimes \mathcal{O}_S (-C) \otimes L^2)$ \smallskip
\item[] \quad \quad \quad  \quad \quad $\rightarrow H^1 (S,\bigwedge^{i-1} \pi ^* \mathcal{M}_\mathcal{E} \otimes M \otimes \mathcal{O}_S (-C) \otimes L^2) \rightarrow \cdots$.
\end{enumerate}
\smallskip

\noindent Note that
\smallskip
\begin{enumerate}
\item[(4.8)] \quad $H^k (S,\bigwedge^{i-1} \pi^* \mathcal{M}_\mathcal{E} \otimes M \otimes \mathcal{O}_S (-C) \otimes L^2) =0$ for all $k \geq 0$
\end{enumerate}
\smallskip

\noindent since $M \otimes \mathcal{O}_S (-C) \otimes L^2=\mathcal{O}_S (-C_0 +3(m-g-1)\mathfrak{f})$. Therefore (4.5), (4.7) and (4.8) show that
\smallskip
\begin{enumerate}
\item[(4.9)] \quad $\beta_{i,2} ~~ = h^1 (S,\bigwedge^i \pi^* \mathcal{M}_\mathcal{E} \otimes \mathcal{O}_S (-C) \otimes L^2)$ \smallskip
\item[] \quad \quad\quad  $= h^1 (S,\bigwedge^i \pi^* \mathcal{M}_\mathcal{E} \otimes \mathcal{O}_S ((2m+b-2g-2)\mathfrak{f}))$ \smallskip
\item[] \quad \quad\quad  $= h^1 ({\mathbb P}^1,\bigwedge^i \mathcal{M}_\mathcal{E} \otimes \mathcal{O}_{{\mathbb P}^1} (2m+b-2g-2))$.
\end{enumerate}
\smallskip

\noindent Since $\mathcal{M}_\mathcal{E} = \mathcal{O}_{{\mathbb P}^1} (-1)^{\oplus (r-1)}$ and $2m+b=2g+1+p$, one can immediately check that $\beta_{i,2} = {{r-1} \choose {i}} h^1 ({\mathbb P}^1,\mathcal{O}_{{\mathbb P}^1} (p-1-i))$. This completes the proof of (b) when $L$ is a very ample line bundle on $S$.

Now we turn to the case $(\alpha)$, i.e. $\mathcal{L}=A^{g+1}$, $S={\mathbb P} (\mathcal{O}_{{\mathbb P}^1} \oplus \mathcal{O}_{{\mathbb P}^1} (-g-1))$ and $L=\mathcal{O}_S (C_0 + (g+1) \mathfrak{f})$. Let $\psi : S \rightarrow {\mathbb P}^{g+2}$ be the map defined by $|L|$, and let $S' = \psi (S)$. Then $S' \subset {\mathbb P}^{g+1}$ is a cone over a rational normal curve of degree $g+1$. Let $\mathcal{I}_{S'}$ be the sheaf of ideals of $S'$, and consider the exact sequence
\smallskip
\begin{enumerate}
\item[(4.10)] \quad \quad \quad $0 \rightarrow \mathcal{I}_{S'} \rightarrow \mathcal{I}_{C} \rightarrow \mathcal{I}_{C/S'} \rightarrow 0$.
\end{enumerate}
\smallskip

\noindent Let $\mathcal{M}_{S'}$ be the restriction of $\mathcal{M}$ to $S'$. By the same method as in (4.3) and (4.4), we have
\smallskip
\begin{enumerate}
\item[(4.11)] \quad \quad \quad $\beta_{i,2} = h^1 (S',\bigwedge^i \mathcal{M}_{S'} \otimes \mathcal{I}_{C/S'} (2))$.
\end{enumerate}
\smallskip

\noindent Since $\psi \upharpoonright_{C} : C \rightarrow C$ is an isomorphism, the exact sequence

\smallskip
\begin{enumerate}
\item[(4.12)] \quad \quad \quad $0 \rightarrow \mathcal{I}_{C/S'} \rightarrow \mathcal{O}_{S'} \rightarrow \mathcal{O}_C \rightarrow 0$.
\end{enumerate}
\smallskip
\noindent is the direct image of
\smallskip
\begin{enumerate}
\item[(4.13)] \quad \quad \quad $0 \rightarrow \mathcal{O}_S (-C) \rightarrow \mathcal{O}_S \rightarrow \mathcal{O}_C \rightarrow 0$.
\end{enumerate}
\smallskip
\noindent by $\psi : S \rightarrow S'$. In particular, note that $\psi_* \mathcal{O}_S (-C) = \mathcal{I}_{C/S'}$. Let $\mathcal{M}_S$ be the kernel of the evaluation map $H^0 (S,L) \otimes \mathcal{O}_S \rightarrow L$. Since $\mathcal{M}_S = \psi^* \mathcal{M}_{S'}$,
\smallskip
\begin{enumerate}
\item[(4.14)] \quad \quad \quad $H^1 (S',\bigwedge^i \mathcal{M}_{S'} \otimes \mathcal{I}_{C/S'} (2)) \cong H^1 (S,\bigwedge^i \mathcal{M}_S \otimes \mathcal{O}_S (-C) \otimes L^2)$.
\end{enumerate}
\smallskip
\noindent Therefore (4.11) and (4.14) imply that
\smallskip
\begin{enumerate}
\item[(4.15)] \quad \quad \quad $\beta_{i,2} = h^1 (S,\bigwedge^i \mathcal{M}_S \otimes \mathcal{O}_S (-C) \otimes L^2)$.
\end{enumerate}
\smallskip

\noindent Now one can compute $\beta_{i,2}$ by the formula in (4.9).

Finally, we consider the case $(\beta)$, i.e. $\mathcal{L} = A^g \otimes \mathcal{O}_C (P)$ for some $P \in C$, $S={\mathbb P} (\mathcal{O}_{{\mathbb P}^1} \oplus \mathcal{O}_{{\mathbb P}^1} (-g))$ and $L=\mathcal{O}_S (C_0 + g \mathfrak{f})$. Let $\psi : S \rightarrow {\mathbb P}^{g+1}$ be the map defined by $|L|$, and let $S' = \psi (S)$. Then $S' \subset {\mathbb P}^{g+1}$ is a cone over a rational normal curve of degree $g$. Since $\psi \upharpoonright_{C} : C \rightarrow C$ is an isomorphism, all the arguments in the proof of $(\alpha)$ are available for this case, and so $\beta_{i,2}$ is determined by the formula in (4.9).

\smallskip

\noindent (c) The Hilbert series of the homogeneous coordinate ring $R_C = R/I_C$ of $C \subset {\mathbb P}^r$ is given by
\begin{equation*}
\Psi_{R_C} (\lambda) = \frac{g \lambda^2 + (g+p) \lambda+1}{(1- \lambda)^2}.
\end{equation*}
On the use of Betti numbers $\beta_{i,j}$ we also may write
\begin{equation*}
\Psi_{R_C} (\lambda) = \frac{1-\beta_{1,1} \lambda^2 +\sum_{k\geq 2} \{ (-1)^k
\beta_{k,1} + (-1)^{k-1} \beta_{k-1,2} \} \lambda^{k+1}}{(1-\lambda)^{r+1}}.
\end{equation*}
Then the formula in (c) is obtained by comparing the coefficients.
\end{proof}

\vskip 1 cm

\section{Minimal Free Resolution II - Low Degree}
\noindent Our next aim is to study the minimal free resolution of a linearly normal hyperelliptic curve $C$ of arithmetic genus $g$ and of degree $d \leq 2g$.

Let $\mathcal{L}$ be a very ample line bundle on $C$ of degree $d \leq 2g$ and with the hyperelliptic factorization type $(m,b)$. Then Theorem~\ref{thm:main1} says that $2 \leq b \leq g+1$, $m+b \geq g+2$, and $\mathcal{L}$ is non-special. Now consider the linearly normal embedding
\begin{equation*}
C \subset {\mathbb P} H^0 (C,\mathcal{L})^* ={\mathbb P}^r,\quad r=d-g.
\end{equation*}
of $C$ defined by $|\mathcal{L}|$. Let $R$ be the homogeneous coordinate ring of ${\mathbb P}^r$, and let $I_C$ be the vanishing ideal of $C$. Let ${\mathbb F}_{\bullet}$ be a minimal free resolution of $I_C$ over $R$:
\begin{equation*}
{\mathbb F}_{\bullet} : 0 \rightarrow F_r \rightarrow \cdots \rightarrow F_i \rightarrow \cdots \rightarrow F_1 \rightarrow I_C \rightarrow 0,
\end{equation*}
where $F_i = \bigoplus_{j \in {\mathbb Z}} R(-i-j)^{\beta_{i,j}}$. We first introduce to some notation:

\smallskip

\begin{definition}
For a very ample line bundle $\mathcal{L}$ on $C$ of degree $d \leq 2g$ with the hyperelliptic factorization type $(m,b)$, let
\smallskip
\begin{enumerate}
\item[$(\alpha)$] $\nu :=\lceil \frac{b-1}{m+b-g-1}\rceil$, \smallskip
\item[$(\beta)$] $\tau :=\lfloor \frac{2g+1-b}{m}\rfloor$, and \smallskip
\item[$(\gamma)$] $p   := ( m+b-g-1 ) \times \nu - b +1$.
\end{enumerate}
\end{definition}
\smallskip

\noindent Clearly, one can read off the integer $p$ from the identity $\nu = \frac{b-1+p}{m+b-g-1}$. Note that $\nu \geq 3$ and $\tau \geq 2$ since we consider the case where $d \leq 2g$.

The following is our main result in this section:

\smallskip
\begin{theorem}\label{thm:lowdegree}
Under the situation just stated, let
\smallskip
\begin{enumerate}
\item[(1)] For all $j \geq 2$,
\begin{equation*}
h^1 ({\mathbb P}^r,\mathcal{I}_C (j))=\sum_{k=0} ^{j-2} h^0 ({\mathbb P}^1,\mathcal{O}_{{\mathbb P}^1} (k(g+1-b)+2g-b-jm )).
\end{equation*}
In particular, $C \subset {\mathbb P}^r$ is $j$-normal if and only if $j \geq \nu$. Therefore $\mbox{reg}(C)=\nu+1$. \smallskip
\item[(2)] The graded Betti numbers $\beta_{i,j}$'s are as follows:
\smallskip
\begin{enumerate}
\item[$(i)$] $\beta_{i,1} = i{{r} \choose {i+1}}$ for all $i \geq 1$, \smallskip
\item[$(ii)$] $\beta_{i,j} = 0$ for $2 \leq j \leq \tau-1$ and all $i \geq 1$,  \smallskip
\item[$(iii)$] $\beta_{1,\nu}=\cdots=\beta_{p,\nu}=0$ and $\beta_{p+1,\nu} \neq 0$, and \smallskip
\item[$(iv)$] $\beta_{r,\nu}=2g+1-d$.
\end{enumerate}
\smallskip
\item[(3)] If $p=0$, then $C \subset {\mathbb P}^r$ fails to satisfy property $N_{\nu,1}$, and if $p \geq 1$, then $C \subset {\mathbb P}^r$ satisfies property $N_{\nu,p}$ while it fails to satisfy property $N_{\nu,p+1}$.
\end{enumerate}
\end{theorem}

\smallskip

\begin{proof}
Let $S={\mathbb P} (\mathcal{O}_{{\mathbb P}^1} \oplus \mathcal{O}_{{\mathbb P}^1} (b-g-1))$ be the rational ruled surface constructed in Theorem \ref{thm:surface}. Recall the following conditions:
\begin{equation*}
\mathcal{O}_S (\mathfrak{f}) \otimes \mathcal{O}_C = A, \quad \mathcal{O}_S (C_0) \otimes \mathcal{O}_C = B, \quad \mbox{and} \quad  C \equiv 2C_0 + (2g+2-b)\mathfrak{f}
\end{equation*}

\noindent Also $\mathcal{L}$ is equal to the restriction of the line bundle $L:=\mathcal{O}_S (C_0 + m \mathfrak{f})$ on $S$ to $C$. Now, a crucial fact is that $L$ is very ample since $m + b \geq g+2$. Therefore $|L|$ defines a smooth rational normal surface scroll $S \subset {\mathbb P}^r$ which contains $C$. Keeping these situation in mind, let us begin with the proof.
\smallskip

\noindent (1) From (4.2), we can prove that
\begin{equation*}
H^1 ({\mathbb P}^r,\mathcal{I}_C (j)) \cong H^1 (S,\mathcal{O}_S (-C) \otimes L^j) \quad \mbox{for all $j \geq 2$}.
\end{equation*}
Since $C \equiv 2C_0 + (2g+2-b) \mathfrak{f}$, we get the desired formula

\smallskip
\begin{enumerate}
\item[] $h^1 ({\mathbb P}^r,\mathcal{I}_C (j)) = h^1 (S,\mathcal{O}_S (-C) \otimes L^j)$ \smallskip
\item[] \quad  \quad \quad\quad \quad \quad
$= h^1 (S,\mathcal{O}_S ((j-2)C_0 + (jm+b-2g-2)))$ \smallskip
\item[] \quad \quad \quad\quad \quad \quad
$= \sum_{k=0} ^{j-2} h^0 ({\mathbb P}^1,\mathcal{O}_{{\mathbb P}^1} (k(g+1-b)+2g-b-jm))$.
\end{enumerate}
\smallskip
Therefore $H^1 ({\mathbb P}^r,\mathcal{I}_C (j))=0$ if and only if
\begin{equation*}
(j-2)(g+1-b)+2g-b-jm \leq -1,
\end{equation*}
or equivalently $j \geq \nu$. This implies that $\mbox{reg}(C)=\nu+1$ since $\mathcal{L}$ is non-special.

\smallskip

\noindent (2) To prove $(i)$ and $(ii)$, we will first show that
\smallskip

\begin{enumerate}
\item[(5.2)] \quad \quad $H^0 ({\mathbb P}^r,\mathcal{I}_S (j)) = H^0 ({\mathbb P}^r,\mathcal{I}_C (j))$ for all $j \leq \tau$.
\end{enumerate}
\smallskip

\noindent Indeed, in the exact sequence
\begin{equation*}
0 \rightarrow H^0 ({\mathbb P}^r,\mathcal{I}_S (j)) \rightarrow H^0 ({\mathbb P}^r,\mathcal{I}_C (j)) \rightarrow H^0 (S,\mathcal{O}_S (-C) \otimes L^j) \rightarrow 0
\end{equation*}
induced by (4.2), the third term
\begin{equation*}
H^0 (S,\mathcal{O}_S (-C) \otimes L^j )=H^0 (S,\mathcal{O}_S ((j-2)C_0 + (j m -2g-2+b)\mathfrak{f}))
\end{equation*}
vanishes if and only if $j m -2g-2+b \leq -1$, or equivalently $j \leq \tau$. Clearly (5.2) implies that for $1 \leq j \leq \tau -1$,
the $j$-th row of the Betti diagram of $C \subset {\mathbb P}^r$ coincides with those of the Betti diagram of $S \subset {\mathbb P}^r$, which completes the proof of $(i)$ and $(ii)$.

For $(iii)$, note that $C \subset {\mathbb P}^r$ is $\nu$-normal and hence
\smallskip
\begin{enumerate}
\item[(5.3)] \quad \quad \quad $\beta_{i,\nu}= h^1 ({\mathbb P}^r,\bigwedge^i \mathcal{M} \otimes \mathcal{I}_C (\nu))$ for all $i \geq 1$.
\end{enumerate}
\smallskip

\noindent Also since $C \subset {\mathbb P}^r$ is $(\nu+1)$-regular, it satisfies property $N_{\nu,p}$ if and only if $\beta_{p,\nu}=0$. By using the same method as in (4.3) and (4.4), we know that

\smallskip
\begin{enumerate}
\item[(5.4)] \quad \quad $\beta_{p,\nu}= h^1 (S,\bigwedge^p \mathcal{M}_S \otimes \mathcal{O}_S (-C) \otimes L^{\nu})$.
\end{enumerate}
\smallskip

\noindent Then by (4.6), the vanishing of $\beta_{p,\nu}$ is verified if

\smallskip
\begin{enumerate}
\item[(5.5)] \quad \quad $H^1 (S,\bigwedge^p \pi^* \mathcal{M}_\mathcal{E} \otimes \mathcal{O}_S (-C) \otimes L^\nu)=0$
\end{enumerate}
\smallskip

\noindent and

\smallskip
\begin{enumerate}
\item[(5.6)] \quad $H^1 (S,\bigwedge^{p-1} \pi^* \mathcal{M}_\mathcal{E} \otimes M \otimes \mathcal{O}_S (-C) \otimes L^\nu)=0$.
\end{enumerate}
\smallskip

\noindent For (5.5), note that $H^1 (S,\bigwedge^p \pi^* \mathcal{M}_\mathcal{E} \otimes \mathcal{O}_S (-C) \otimes L^\nu)$ is isomorphic to a direct sum of copies of cohomology groups of the form

\smallskip
\begin{enumerate}
\item[(5.7)] \quad $H^1 ({\mathbb P}^1 , \mathcal{O}_{{\mathbb P}^1} (\nu m -p - (k+2)(g+1-b) - b))$
\end{enumerate}
\smallskip

\noindent with $0 \leq k \leq \nu-2$. Since
\begin{equation*}
\nu m -p - (k+2)(g+1-b) - b \geq \nu m -p - \nu (g+1-b) - b =-1,
\end{equation*}
the cohomology group in (5.7) vanishes, which completes the proof of (5.5). By the same way, we can check (5.6), and so $C \subset {\mathbb P}^r$ satisfies property $N_{\nu , p}$.

Now it remains to show that
property $N_{\nu,p+1}$ fails to hold, or equivalently, $\beta_{p+1, \nu} \neq 0$. To this aim, we use the finite scheme $\Gamma = C \cap C_0$. Remember that $\mathcal{O}_C (\Gamma)=B$. Indeed, we will show that $\Gamma$ is a geometric obstruction for property $N_{\nu,p+1}$ of $C \subset {\mathbb P}^r$. Since $C \subset {\mathbb P}^r$ is $\nu$-normal, we have
\smallskip
\begin{enumerate}
\item[(5.8)] \quad \quad \quad $\beta_{p+1,\nu}= h^1 ({\mathbb P}^r,\bigwedge^{p+1} \mathcal{M} \otimes \mathcal{I}_C (\nu))$.
\end{enumerate}
\smallskip

\noindent Let $\mathcal{I}_{\Gamma}$ be the sheaf of ideals of $\Gamma \subset {\mathbb P}^r$, and consider the exact sequence
\smallskip
\begin{enumerate}
\item[(5.9)] \quad \quad \quad $0 \rightarrow \mathcal{I}_C \rightarrow \mathcal{I}_{\Gamma} \rightarrow B^{-1} \rightarrow 0$.
\end{enumerate}
\smallskip

\noindent This induces the cohomology long exact sequence

\begin{enumerate}
\item[(5.10)] \quad  $H^1 ({\mathbb P}^r,\bigwedge^{p+1} \mathcal{M} \otimes \mathcal{I}_C (\nu)) \rightarrow H^1 ({\mathbb P}^r,\bigwedge^{p+1} \mathcal{M} \otimes \mathcal{I}_{\Gamma} (\nu))$ \smallskip
\item[] \quad \quad \quad \quad \quad \quad \quad \quad $ \rightarrow H^1 (C,\bigwedge^{p+1} \mathcal{M}_C \otimes \mathcal{L}^\nu \otimes B^{-1})$
\end{enumerate}
\smallskip

\noindent where $\mathcal{M}_C$ is the restriction of $\mathcal{M}$ to $C$. By (5.8) and (5.10), $\beta_{p+1,\nu} \neq 0$ if the second term of (5.10) is nonzero while the third term is zero.

Since $\mathcal{O}_C$ is $2$-regular as a coherent sheaf on ${\mathbb P}^r$ and $\nu \geq 3$, we have
\begin{equation*}
H^1 (C,\bigwedge^{p+1} \mathcal{M}_C \otimes \mathcal{L}^{\nu-1})=0.
\end{equation*}
This implies that $H^1 (C,\bigwedge^{p+1} \mathcal{M}_C \otimes \mathcal{L}^\nu \otimes B^{-1})=0$ since $\mathcal{L} \otimes B^{-1} = A^m$ is an effective divisor on $C$.

We turn to the proof of the non-vanishing of $H^1 ({\mathbb P}^r,\bigwedge^{p+1} \mathcal{M} \otimes \mathcal{I}_{\Gamma} (\nu))$. Let $\Lambda = {\mathbb P}^{m+b-g-1}$ be the linear span of $C_0$. Let $\mathcal{I}_{\Gamma/\Lambda} \subset \mathcal{O}_{\Lambda}$ be the sheaf of ideals of $\Gamma \subset \Lambda$, and let $\mathcal{I}_{\Lambda/{\mathbb P}^r} \subset \mathcal{O}_{{\mathbb P}^r}$ be the sheaf of ideals of $\Lambda \subset {\mathbb P}^r$. Then from the exact sequence $0 \rightarrow \mathcal{I}_{\Lambda}  \rightarrow \mathcal{I}_{\Gamma} \rightarrow \mathcal{I}_{\Gamma/\Lambda} \rightarrow 0$, one can show that
\begin{equation*}
(5.11) \quad \quad H^1 ({\mathbb P}^r,\bigwedge^{p+1} \mathcal{M} \otimes \mathcal{I}_{\Gamma} (\nu)) \cong H^1 (\Lambda,\bigwedge^{p+1} \mathcal{M}' \otimes \mathcal{I}_{\Gamma/\Lambda} (\nu)) \quad
\end{equation*}
where $\mathcal{M}'$ is the restriction of $\mathcal{M}$ to $\Lambda$. Note that $\mathcal{M}_{\Lambda} := \Omega_{\Lambda} (1)$ is a direct summand of $\mathcal{M}'$, and hence
\begin{equation*}
(5.12) \quad \quad H^1 (\Lambda,\bigwedge^{p+1} \mathcal{M}_{\Lambda} \otimes \mathcal{I}_{\Gamma/\Lambda} (\nu)) \subset H^1 (\Lambda,\bigwedge^{p+1} \mathcal{M}' \otimes \mathcal{I}_{\Gamma/\Lambda} (\nu)).
\end{equation*}
Since $\mbox{length}(\Gamma)=b=\nu (m+b-g-1)+1-p$ and $\Gamma$ is contained in the rational normal curve $C_0 \subset \Lambda$, Lemma \ref{lem:finiteonrnc} below says that
\begin{equation*}
H^1 (\Lambda,\bigwedge^{p+1} \mathcal{M}_{\Lambda} \otimes \mathcal{I}_{\Gamma/\Lambda} (\nu)) \neq 0
\end{equation*}
In conclusion, $\beta_{p,\nu}  >0$ and so $C \subset {\mathbb P}^r$ fails to satisfy property $N_{\nu,p+1}$.

For $(iv)$, note that
\begin{equation*}
\beta_{r,\nu}= h^1 (S,\bigwedge^r \mathcal{M}_S \otimes \mathcal{O}_S (-C) \otimes L^{\nu})
\end{equation*}
by the same method as in (4.3) and (4.4). Since $\bigwedge^r \mathcal{M} = L^{-1}$, we have
\begin{equation*}
\beta_{r,\nu}= h^1 (S, \mathcal{O}_S (-C) \otimes L^{\nu-1})=2g+1-d.
\end{equation*}

\smallskip
\noindent (3) The assertion comes directly from (2).
\end{proof}

\smallskip

\begin{lemma}\label{lem:finiteonrnc}
Let $Y \subset {\mathbb P}^n$ be a rational normal curve of degree $n$, and let $\Gamma \subset Y$ be a finite subscheme of length $\ell =\nu n +t$ for some $\nu \geq 2$ and $-n+2 \leq t \leq 1$.
Then for the sheaf of ideals $\mathcal{I}_{\Gamma / {\mathbb P}^n}$ of $\Gamma \subset {\mathbb P}^n$ and $\mathcal{M}_{{\mathbb P}^n} := \Omega_{{\mathbb P}^n} (1)$,
\begin{equation*}
H^1 ({\mathbb P}^n,\bigwedge^{2-t} \mathcal{M}_{{\mathbb P}^n} \otimes \mathcal{I}_{\Gamma/{\mathbb P}^n} (\nu)) \neq 0.
\end{equation*}
\end{lemma}

\smallskip

\begin{proof}
Let $\mathcal{M}_Y$ be the restriction of $\mathcal{M}_{{\mathbb P}^n}$ to $Y$. Remember that $Y \subset {\mathbb P}^n$ satisfies property $N_p$ for all $p \geq 0$. Thus the exact sequence
\begin{equation*}
0 \rightarrow \mathcal{I}_Y \rightarrow \mathcal{I}_{\Gamma} \rightarrow \mathcal{O}_{{\mathbb P}^1} (-\ell) \rightarrow 0
\end{equation*}
induces the isomorphisms

\smallskip

\begin{enumerate}
\item[(5.13)] \quad $H^1 ({\mathbb P}^n,\bigwedge^i \mathcal{M}_{{\mathbb P}^n} \otimes \mathcal{I}_{\Gamma/{\mathbb P}^n} (j)) \cong H^1 ({\mathbb P}^1,\bigwedge^i \mathcal{M}_Y  \otimes \mathcal{O}_{{\mathbb P}^1} (nj-\ell))$
\end{enumerate}

\smallskip

\noindent for all $i \geq 0$ and $j \geq 2$. Since $\mathcal{M}_Y = \mathcal{O}_{{\mathbb P}^1} (-1) ^{\oplus n}$, we have

\smallskip

\begin{enumerate}
\item[(5.14)] $h^1 ({\mathbb P}^1,\bigwedge^i \mathcal{M}_Y  \otimes \mathcal{O}_{{\mathbb P}^1} (nj-\ell)) = {{n} \choose {i}} h^1 ({\mathbb P}^1,\mathcal{O}_{{\mathbb P}^1} (-i+nj-\ell))$
\end{enumerate}

\smallskip

\noindent In particular, for $i=2-t$ and $j = \nu$ we have

\smallskip

\begin{enumerate}
\item[(5.15)] $h^1 ({\mathbb P}^1,\bigwedge^{2-t} \mathcal{M}_Y  \otimes \mathcal{O}_{{\mathbb P}^1} (n \nu -\ell)) = {{n} \choose {2-t}} h^1 ({\mathbb P}^1,\mathcal{O}_{{\mathbb P}^1} (-2))$ \smallskip
\item[] \quad \quad \quad \quad \quad \quad \quad \quad \quad  \quad \quad \quad \quad \quad $= {{n} \choose {2-t}} >0$.
\end{enumerate}

\smallskip

\noindent The proof is completed by combining (5.13) and (5.15).
\end{proof}

\vskip 1 cm

\section{Minimal Free Resolution III - Geometric Meaning}
\noindent Let $C$ be a hyperelliptic curve of arithmetic genus $g$, and let $\mathcal{L}$ be a very ample line bundle on $C$ of degree $d$ and with the factorization type $(m,b)$. In the previous two sections, we study the minimal free resolution of the linearly normal curve
\begin{equation*}
C \subset {\mathbb P} H^0 (C,\mathcal{L})^* = {\mathbb P}^r , \quad r= d-g,
\end{equation*}
embedded by $|\mathcal{L}|$. This section will be devoted to explain some related geometric meaning of our results. We will deal with the case where $d \geq 2g+1$ and the case where $d \leq 2g$ in turn. \\

\noindent {\bf 6.1. High Degree}

\smallskip

\noindent If $d = 2g+1+p$, $p \geq 0$, then $C \subset {\mathbb P}^r$ satisfies property $N_p$ while it fails to satisfy property $N_{p+1}$. Also the graded Betti numbers are precisely determined by $d$ (Theorem~\ref{thm:main3}). In particular, they are independent on $(m,b)$, and so one cannot determine $(m,b)$ from the minimal free resolution. On the other hand, $(m,b)$ gives us the precise numerical type of the rational ruled surface scroll
\begin{equation*}
S' = \bigcup_{y \in {\mathbb P}^1} <f^{-1} (y)> \subset {\mathbb P}^r
\end{equation*}
where $f : C \rightarrow {\mathbb P}^1$ is the double covering. And in Theorem~\ref{thm:main3}, the graded Betti numbers of $C$ are calculated by using $S'$. Therefore $S'$ is a geometric obstruction of property $N_{p+1}$ of $C \subset {\mathbb P}^r$.

In this subsection, we will show that the failure of property $N_{p+1}$ of $C \subset {\mathbb P}^r$ can be explained by a specific finite subscheme of $C$. Let $\mathcal{L} = A^m \otimes B$ be the hyperelliptic factorization of $\mathcal{L}$. Then the followings are equivalent:
\smallskip
$$(6.1) \quad \begin{cases} 1. \quad H^0 (C,\mathcal{L} \otimes \omega_C ^{-1} ) = H^0 (C,A^{m-g+1} \otimes B) \neq 0 \smallskip \\
2. \quad \mbox{There exists a $(p+3)$-secant $(p+1)$-plane to $C$.} \\
3. \quad m \geq g-1
\end{cases}$$
\smallskip
\noindent Let us recall Theorem 1.1 in \cite{EGHP}, which asserts that if a non-degenerate closed subscheme $X \subset {\mathbb P}^N$ admits a $(p+3)$-secant $(p+1)$-plane, then it fails to satisfy the condition $N_{2,p+1}$. Therefore if $m \geq g-1$, then $C \subset {\mathbb P}^r$ fails to satisfy property $N_{p+1}$ because there is a $(p+3)$-secant $(p+1)$-plane to $C$. More precisely, any effective divisor $D \in |A^{m-g+1} \otimes B|$ spans a $(p+3)$-secant $(p+1)$-plane to $C$. Geometrically, $\langle D \rangle$ is the span of $\Lambda$ and general $(m+1-g)$ rulings of $S'$.

From now on, we focus on the case where $m \leq g-2$. Since $B$ is normalized and $m-g+1 <0$, we have $H^0 (C,\mathcal{L} \otimes \omega_C ^{-1} ) = H^0 (C,A^{m-g+1} \otimes B) = 0$. Thus $C \subset {\mathbb P}^r$ does not admit a $(p+3)$-secant $(p+1)$-plane by (6.1). Note that $S'=S$ is smooth. Let $C_0$ be the minimal section of $S$. The finite scheme $\Gamma := C \cap C_0$ is of length $b=(p+g-m)+(g+1-m)$ and spans the linear space $\Lambda := \langle C_0 \rangle$ of dimension $(p+g-m)$. That is,
\smallskip
\begin{enumerate}
\item[(6.2)] $\Lambda= \langle \Gamma \rangle$ is a $\{(p+g-m)+(g+1-m)\}$-secant $(p+g-m)$-plane to $C$
\end{enumerate}
\smallskip
where $g+1-m \geq 3$, and so $C \subset {\mathbb P}^r$ fails to satisfy property $N_{p+g-m}$ by Theorem 1.1 in \cite{EGHP}. Now we will see that $\Gamma$ is a geometric obstruction of property $N_{p+1}$ of $C$. A crucial fact is that $\Gamma$ is contained in the rational normal curve $C_0$.
\smallskip

\begin{theorem}\label{thm:geometricobstruction}
$\Gamma$ obstructs property $N_{p+1}$ of $C \subset {\mathbb P}^r$.
\end{theorem}

\smallskip

\begin{proof}
Consider the cohomology long exact sequence (5.10) for $\nu =2$:
\smallskip
\begin{enumerate}
\item[(6.3)] \quad  $H^1 ({\mathbb P}^r,\bigwedge^{p+1} \mathcal{M} \otimes \mathcal{I}_C (2)) \rightarrow H^1 ({\mathbb P}^r,\bigwedge^{p+1} \mathcal{M} \otimes \mathcal{I}_{\Gamma} (2))$ \smallskip
\item[] \quad \quad \quad \quad \quad \quad \quad \quad $ \rightarrow H^1 (C,\bigwedge^{p+1} \mathcal{M}_C \otimes \mathcal{L} \otimes A^m)$
\end{enumerate}
\smallskip

\noindent Since $\mathcal{L}$ is non-special and $(C,\mathcal{L})$ satisfies property $N_p$,
\smallskip
\begin{enumerate}
\item[(6.4)] \quad \quad  $\beta_{p,1} = h^1 (C,\bigwedge^{p+1} \mathcal{M}_C \otimes \mathcal{L})=0$.
\end{enumerate}
\smallskip

\noindent Clearly this implies that

\smallskip
\begin{enumerate}
\item[(6.5)] \quad \quad  $H^1 (C,\bigwedge^{p+1} \mathcal{M}_C \otimes \mathcal{L} \otimes A^m) = 0$.
\end{enumerate}
\smallskip

\noindent To prove the non-vanishing of the middle term in (6.3), we apply Lemma~\ref{lem:finiteonrnc} to $\Gamma$.
Since $|\Gamma|=2(p+g-m)+(1-p)$, Lemma~\ref{lem:finiteonrnc} shows that
\begin{equation*}
(6.6) \quad \quad H^1 (\Lambda,\bigwedge^{p+1} \mathcal{M}_{\Lambda} \otimes \mathcal{I}_{\Gamma/\Lambda} (2)) \neq 0 \quad \quad \quad \quad \quad \quad \quad \quad \quad \quad
\end{equation*}
where $\mathcal{M}_{\Lambda} = \Omega_{\Lambda}(1)$ and $\mathcal{I}_{\Gamma/\Lambda}$ is the sheaf of ideals of $\Gamma \subset \Lambda$. By the same argument as in (5.11) and (5.12), one can show that (6.6) implies
\begin{equation*}
(6.7) \quad \quad H^1 ({\mathbb P}^r,\bigwedge^{p+1} \mathcal{M} \otimes \mathcal{I}_{\Gamma} (2)) \neq 0. \quad \quad \quad \quad \quad \quad \quad \quad \quad \quad \quad
\end{equation*}
In conclusion, $\beta_{p+1,1} = h^1 ({\mathbb P}^r,\bigwedge^{p+1} \mathcal{M} \otimes \mathcal{I}_C (2)) \neq 0$ by (6.3), (6.5) and (6.7), and hence $C \subset {\mathbb P}^r$ fails to satisfy property $N_{p+1}$.
\end{proof}

\smallskip

\begin{remark}
Let $C$ be a smooth projective curve of genus $g \geq 1$, and let $\mathcal{L}$ be a line bundle on $C$ of degree $d=2g+1+p$. Theorem 2 in \cite{GL2} says that $(C,\mathcal{L})$ fails to satisfy property $N_{p+1}$ if and only if either $C$ is a hyperelliptic curve or $C \subset {\mathbb P} H^0 (C,\mathcal{L})^*$ admits a $(p+3)$-secant $(p+1)$-plane. By our investigation in this subsection, every hyperelliptic curve $C$ admit a finite subscheme $\Gamma \subset C$ which obstructs property $N_{p+1}$ of $(C,\mathcal{L})$. \qed
\end{remark}

\smallskip

\noindent {\bf 6.2. Low Degree}

\smallskip

\noindent Now we consider the case where $\mathcal{L}$ is a very ample line bundle of degree $d \leq 2g$ and with the factorization type $(m,b)$. Then Theorem~\ref{thm:lowdegree} illustrates how $(m,b)$ effects on the shape of the minimal free resolution of $C \subset {\mathbb P}^r,~r=d-g$. More precisely, the Castelnuovo-Mumford regularity of $C$ is equal to $\nu+1$ where $\nu = \frac{b-1+p}{m+b-g-1}$. Also $C$ satisfies the condition $N_{\nu,p}$ while it fails to satisfy the condition $N_{\nu,p+1}$. Thus one can read off the values of $\nu$ and $p$ from the minimal free resolution of $C$. Then the identities
\begin{equation*}
(6.8) \quad \quad d=2m+b \quad \mbox{and} \quad \nu = \frac{b-1+p}{m+b-g-1} \quad \quad
\end{equation*}
enables us to determine the pair $(m,b)$ as follows:
\begin{equation*}
(6.9) \quad \quad m \quad = \quad d-g-1 - \frac{2g+1+p-d}{\nu-2} \quad \quad \quad
\end{equation*}
\begin{equation*}
(6.10) \quad \quad b \quad = \quad 2g+2-d+\frac{2(2g+1+p-d)}{\nu -2} \quad \quad
\end{equation*}

\noindent Indeed, the factorization type interacts with the minimal free resolution more strongly in the following sense:

\begin{theorem}\label{thm:DistinctBettiTable}
Let $C_1$ and $C_2$ be hyperelliptic curves of the same arithmetic genus $g$. For $i =1,2$, let $\mathcal{L}_i$ be a very ample line bundles on $C_i$ of degree $d \leq 2g$ and with the factorization types $(m_i , b_i )$. Then $C_1 \subset {\mathbb P} H^0 (C_1,\mathcal{L}_1)^*$ and $C_2 \subset {\mathbb P} H^0 (C_2 ,\mathcal{L}_2)^*$ have the same graded Betti numbers if and only if $(m_1 , b_1 )=(m_2 , b_2)$.
\end{theorem}

\begin{proof}
$(\Longrightarrow)$ Clear from (6.9) and (6.10).

\smallskip

\noindent $(\Longleftarrow)$ Let $(m,b)=(m_1 , b_1)$ and let $\mathcal{L}$ be a line bundle with the factorization type $(m,b)$. We keep the notations in the proof of Theorem~\ref{thm:lowdegree}. From (4.2), we have the following short exact sequence of $R$-modules:
\begin{equation*}
(6.11) \quad \quad \quad 0 \rightarrow I_S \rightarrow I_C \rightarrow E \rightarrow 0 \quad \quad \quad \quad \quad \quad \quad \quad
\end{equation*}
where $E := \bigoplus_{j \in {\mathbb Z}} H^0 (S,\mathcal{O}_S (-C) \otimes L^j )$. Then (6.11) induces the following long exact sequence:
\begin{equation*}
(6.12) \cdots \rightarrow \mbox{Tor}_i ^R (I_S , K)_{i+j} \rightarrow \mbox{Tor}_i ^R (I_C , K )_{i+j} \rightarrow \mbox{Tor}_i ^R (E , K)_{i+j} \quad \quad \quad \quad \quad \quad
\end{equation*}
\begin{equation*}
 \quad \quad \quad  \rightarrow \mbox{Tor}_{i-1} ^R (I_S , K)_{i+j+1} \rightarrow \mbox{Tor}_{i-1} ^R (I_C , K)_{i+j+1} \rightarrow \mbox{Tor}_{i-1} ^R (E , K)_{i+j+1} \rightarrow \cdots
\end{equation*}
By Theorem~\ref{thm:lowdegree}.(2).(i) and (ii) or the proof of them,
\smallskip
\begin{enumerate}
\item[(6.13)] \quad $\mbox{Tor}_i ^R (I_S , K)_{i+1} \cong \mbox{Tor}_i ^R (I_C , K)_{i+1}$ for all $i \geq 1$.
\end{enumerate}
\smallskip
Also since $S \subset {\mathbb P}^r$ is $2$-regular, we have
\smallskip
\begin{enumerate}
\item[(6.14)] \quad $\mbox{Tor}_i ^R (I_S , K)_{i+j} = 0$ for all $i \geq 1$ and $j \geq 2$.
\end{enumerate}
\smallskip
Now (6.12), (6.13) and (6.14) enable us to show that
\smallskip
\begin{enumerate}
\item[(6.15)] \quad $\mbox{Tor}_i ^R (I_C , K)_{i+j} \cong \mbox{Tor}_i ^R (E, K)_{i+j}$ for all $i \geq 1$ and $j \geq 2$.
\end{enumerate}
\smallskip
Note that for any $\mathcal{L}$ with the factorization type $(m,b)$, we have the same pair ($S$, $L$, $\mathcal{O}_S (-C)$) and the isomorphisms in (6.15). Since the $R$-module $E$ is defined by ($S$, $L$, $\mathcal{O}_S (-C)$), (6.13) and (6.15) completes the proof.
\end{proof}
\smallskip

\begin{remark}\label{rem:distinctBettitable}
For a given $d \in \{ g+3 , g+4 , \cdots , 2g \}$, the possible factorization types $(m,b)$ of a very ample line bundle of degree $d$ are as follows:
\smallskip
\begin{enumerate}
\item[1.] If $d=2k$, then $(m,b) \in \{ (k-i , 2i)~ | ~ g+2-k \leq i \leq \lfloor \frac{g+1}{2} \rfloor$.
In particular, there are exactly $(k - \lceil \frac{g+1}{2}\rceil)$ distinct factorization types. \smallskip
\item[2.] If $d=2k+1$, then $(m,b) \in \{ (k-i , 2i+1)~ | ~ g+1-k \leq i \leq \lfloor \frac{g}{2} \rfloor$.
In particular, there are exactly $(k - \lceil \frac{g}{2}\rceil)$ distinct factorization types.
\end{enumerate}
\smallskip

\noindent Therefore Theorem~\ref{thm:DistinctBettiTable} shows that there exist precisely $(k - \lceil \frac{g+1}{2}\rceil)$ distinct Betti diagrams of linearly normal hyperelliptic curves of arithmetic genus $g$ and of degree $2k$, and precisely $(k - \lceil \frac{g}{2}\rceil)$ distinct Betti diagrams of linearly normal hyperelliptic curves of arithmetic genus $g$ and of degree $2k+1$.  \qed
\end{remark}

\vskip 1 cm

\section{Examples}
\noindent In this section we provide some examples which illustrate the results in the present paper. Let $C$ be a hyperelliptic curve of arithmetic genus $g$ and $\mathcal{L}$ a very ample line bundle on $C$ of degree $d$ and with the factorization type $(m,b)$.

\smallskip
\begin{example}
If $d=g+3$, then $(m,b)=(1,g+1)$ by Theorem \ref{thm:main1}.$(3)$. Therefore $\nu =g$, $p=0$ and $\tau =g$. Theorem \ref{thm:lowdegree} shows that
$$h^1 ({\mathbb P}^3,\mathcal{I}_C (j)) = \begin{cases}
(j-1)(g-j) & \quad \mbox{for $1 \leq j \leq g$, and} \\
         0 & \quad \mbox{for all $j \geq g+1$}.
\end{cases}$$
Also $I_C$ is generated by forms of degree $\leq g+1$ while it fails to be generated by forms of degree $\leq g$. Indeed, $I_C$ is generated by $(I_C)_2$, $(I_C)_g$, and $(I_C)_{g+1}$ since $\tau =g$. \qed
\end{example}
\smallskip
\begin{example}
If $d=g+4$, then $(m,b)=(2,g)$, $\nu =g-1$ and $p=0$. Also
\begin{equation*}
h^1 ({\mathbb P}^4,\mathcal{I}_C (j))=\sum_{k=0} ^{j-2} h^0 ({\mathbb P}^1,\mathcal{O}_{{\mathbb P}^1} (g-2j+k)) \quad \mbox{for all} \quad j \geq 2.
\end{equation*}
The vanishing ideal $I_C$ is generated by forms of degree $\leq g$ while it cannot be generated by forms of degree $\leq g-1$. \qed
\end{example}
\smallskip
\begin{example}
If $d=2g$, then $\nu =3$, $h^1 ({\mathbb P}^g,\mathcal{I}_C (2))= 1$ and $(m,b)=(g-i,2i)$ for some $2 \leq i \leq \frac{g+1}{2}$. In this case, $p=i-2$ and hence property $N_{3,i-2}$ holds but property $N_{3,i-1}$ fails to hold for $C \subset {\mathbb P}^g$. In particular, if $(m,b)=(g-2,4)$ then $I_C$ should have quartic generators while if $(m,b)=(g-i,2i)$ for $3 \leq i \leq \frac{g+1}{2}$ then $I_C$ is generated by quadratic and cubic equations. \qed
\end{example}
\smallskip
\begin{example}\label{ex:table}
To illustrate how the algebraic properties of $C \subset {\mathbb P} H^0 (C,\mathcal{L})^*$ vary depending on $(m,b)$, we consider the case where $g=10$ and $13 \leq d \leq 20$. The previous examples explain the cases $d=13, 14$ and $20$. For the remaining cases, we tabulate the information obtained by Theorem~\ref{thm:lowdegree}. See Table 1 below. To simplify notation, we denote $h^1 ({\mathbb P}^r,\mathcal{I}_C (j))$ by $\gamma_j$. \qed
\end{example}

\begin{table}[hbt]
\begin{center}
\begin{tabular}{|c|c|c|c|c|c|c|c|c|c|c|}\hline
$d$  & $(m,b)$ & $\nu$ & $p$ & $\tau$ & $\gamma_2$ & $\gamma_3$ & $\gamma_4$ & $\gamma_5$ & $\gamma_6$ & $\gamma_7$ \\\hline
$15$ & $(3,9)$ & $8$   & $0$ & $4$ &$6$ & $8$ & $6$ & $4$ & $2$ & $1$ \\\cline{2-11}
     & $(2,11)$ & $5$ & $0$ & $5$& $6$ & $8$ & $7$ & $0$ & $0$ & $0$ \\\hline
$16$ & $(4,8)$ & $7$ & $0$ & $3$& $5$ & $4$ & $3$ & $2$ & $1$ & $0$ \\\cline{2-11}
     & $(3,10)$ & $5$ & $1$ & $3$& $5$ & $5$ & $1$ & $0$ & $0$ & $0$ \\\hline
     & $(5,7)$ & $6$ & $0$ & $2$& $4$ & $3$ & $2$ & $1$ & $0$ & $0$ \\\cline{2-11}
$17$ & $(4,9)$ & $4$ & $0$ & $3$& $4$ & $2$ & $0$ & $0$ & $0$ & $0$ \\\cline{2-11}
     & $(3,11)$ & $4$ & $2$ & $3$ & $4$ & $2$ & $0$ & $0$ & $0$ & $0$ \\\hline
     & $(6,6)$ & $5$ & $0$ & $2$ & $3$ & $2$ & $1$ & $0$ & $0$ & $0$ \\\cline{2-11}
$18$ & $(5,8)$ & $4$ & $1$ & $2$ & $3$ & $1$ & $0$ & $0$ & $0$ & $0$ \\\cline{2-11}
     & $(4,10)$ & $3$ & $0$ & $2$ & $3$ & $0$ & $0$ & $0$ & $0$ & $0$ \\\hline
     & $(7,5)$ & $4$ & $0$ & $2$ & $2$ & $1$ & $0$ & $0$ & $0$ & $0$ \\\cline{2-11}
$19$ & $(6,7)$ & $3$ & $0$ & $2$ & $2$ & $0$ & $0$ & $0$ & $0$ & $0$ \\\cline{2-11}
     & $(5,9)$ & $3$ & $1$ & $2$ & $2$ & $0$ & $0$ & $0$ & $0$ & $0$ \\\cline{2-11}
     & $(4,11)$ & $3$ & $2$ & $2$ & $2$ & $0$ & $0$ & $0$ & $0$ & $0$ \\\hline
\end{tabular}
\end{center}
\caption{Hyperelliptic curves of genus $10$}
\end{table}

\bibliographystyle{plain}

\vskip 2 cm

\author{
 \begin{tabular}{ll}
         Department of Mathematics \\
         Korea University \\
         Seoul 136-701 \\
         Republic of Korea \\
		 {\it email: } euisungpark@korea.ac.kr
 \end{tabular}
}

\end{document}